\documentclass{amsart}
\usepackage{latexsym,amssymb,amsthm,amsmath}

\usepackage{amssymb}
\usepackage{amsmath}

\theoremstyle{plain}
\newtheorem{theorem}{Theorem}[section]
\newtheorem*{Theorem B}{Theorem B}
\newtheorem*{Theorem A}{Theorem A}
\newtheorem{lemma}{Lemma}[section]
\newtheorem{proposition}{Proposition}[section]
\newtheorem{corollary}{Corollary}[section]

\newtheorem{definition}{Definition}[section]
\newtheorem{example}{Example}[section]
\numberwithin{equation}{section}

\theoremstyle{remark}

\begin{document}
\title[Warped product skew CR-submanifolds of Kenmotsu manifolds]
{Another class of warped product skew CR-submanifolds of Kenmotsu manifolds }
\author[S. K. Hui, T. Pal and J. Roy]{Shyamal Kumar Hui$^{*}$, Tanumoy Pal and Joydeb Roy}
\keywords{Warped product, skew CR-submanifolds, Kenmotsu manifolds,\\
$*$ Corresponding author}
\begin{abstract}
Recently, Naghi et al. \cite{NAGHI} studied warped product skew CR-submanifold of the form $M_1\times_fM_\bot$ of order $1$ of a Kenmotsu manifold $\bar{M}$ such that $M_1=M_T\times M_\theta$, where $M_T$, $M_\bot$ and $M_\theta$ are invariant, anti-invariant and proper slant submanifolds of $\bar{M}$. The present paper deals with the study of warped product submanifolds by interchanging the two factors $M_T$ and $M_\bot$, i.e, the warped products of the form $M_2\times_fM_T$ such that $M_2=M_\bot\times M_\theta$. The existence of such warped product is ensured by an example and then we characterize such warped product submanifold. A lower bounds of the square norm of second fundamental form is derived with sharp relation, whose equality case is also considered.
\end{abstract}
\maketitle
\section{Introduction}
In 1986, Bejancu \cite{BEJ} introduced the notion of CR-Submanifolds. This family of submanifolds was generalized by Chen \cite{CHENS} as slant submanifolds. Then a more generalization is given as semi-slant submanifolds by Papaghiuc \cite{PAPA}. Next, Cabrerizo et al. \cite{CAR1} defined and studied bi-slant submamifolds and simultanously gave the notion of pseudo-slant submanifolds. The contact version of slant, semi slant and pseudo-slant submanifolds are studied in \cite{LOTTA}, \cite{CAR1} and \cite{6}, respectively. As a generalization of all these class of submanifolds, Ronsse \cite{RONSSE} introduced the notion of skew CR-submanifolds of Kaehler manifolds.\\
\indent The notion of warped product was introduced by Bishop and O'Neill in \cite{BISHOP} to construct the examples of manifolds with negative curvature. The study of warped product submanifolds was initiated by Chen (\cite{CHENCR1}, \cite{CHENCR2}). Then several authors studied warped product submanifolds. For detailed study of warped product submanifolds, we may refer to (\cite{CHENBOOK}, \cite{HAN}-\cite{HUOM}, \cite{UD1}). In this connection it may be mentioned that warped product submanifolds of Kenmotsu manifold are studied in (\cite{ATCE1}-\cite{OTHMAN}, \cite{KHAN}-\cite{KHANS1}, \cite{UDD2}, \cite{UDD1}-\cite{UOO}).\\
\indent Warped product skew CR-submanifolds of Kaehler manifold was studied by Sahin \cite{SAHAIN} and in  \cite{HAIDER} Haider et al. studied this class of submanifolds in cosympletic ambient. Recently Naghi et al. \cite{NAGHI} studied warped product skew CR-submanifolds of the form $M_1\times_fM_\bot$ of order $1$ of a Kenmotsu manifold $\bar{M}$ such that $M_1=M_T\times M_\theta$, where $M_T$, $M_\bot$ and $M_\theta$ are invariant, anti-invariant and proper slant submanifolds of $\bar{M}$. In this paper we have concentrated on another class of warped product skew CR-submanifolds of Kenmotsu manifolds of the form $M_2\times_f M_T$, where $M_2=M_\perp\times M_\theta$. The present paper is organized as follows: in section 2, some preliminaries are given, section 3 is dedicated to the study of skew CR-submanifold of Kenmotsu manifold, in section 4, we provide an example of warped product skew CR-submanifolds of the form $M_2\times_f M_T$ and some basic results of such type of submanifolds are obtained, a characterization of skew CR-warped product of the form $M_2\times_fM_T$ is obtained in section 5. In section 6, we have established two inequalities on a warped product skew CR-submanifold $M=M_2\times_f M_T$ of a Kenmotsu manifold $\bar{M}$.
\section{Preliminaries}
In \cite{TANNO} Tanno classified connected almost contact metric manifolds whose automorphism groups possess the maximum dimension.
For such a manifold, the sectional curvature of plane sections containing $\xi$ is a constant, say $c$. He proved that they could
be divided into three classes:
(i) homogeneous normal contact Riemannian manifolds with $c > 0$,
(ii) global Riemannian products of a line or a circle with a K\"{a}hler manifold of constant holomorphic sectional curvature
          if $c=0$ and
(iii) a warped product space $\mathbb{R} \times _f \mathbb{C}^n$ if $c< 0$.\\
 Kenmotsu \cite{KEN} characterized the differential geometric properties of the manifolds of class (iii) which are nowadays called Kenmotsu
manifolds and later studied by several authors (\cite{HUI1}-\cite{HUI3}) etc.\\
An odd dimensional smooth manifold $\bar{M}^{2m+1}$ is said to be an almost contact metric manifold \cite{BLAIR} if it admits a $(1,1)$
tensor field $\phi$, a vector field $\xi$, an $1$-form $\eta$ and a Riemannian metric $g$ which satisfy
\begin{equation}\label{2.1}
  \phi \xi=0,\ \ \ \eta(\phi X)=0, \ \ \ \phi^2 X=-X+\eta(X)\xi,
\end{equation}
\begin{equation}\label{2.2}
  g(\phi X,Y)=-g(X,\phi Y), \ \ \ \eta(X)=g(X,\xi), \ \ \ \eta(\xi)=1,
\end{equation}
\begin{equation}\label{2.3}
  g(\phi X,\phi Y)=g(X,Y)-\eta(X)\eta(Y)
\end{equation}
for all vector fields $X,Y$ on $\bar{M}$.\\
\indent An almost contact metric manifold $\bar{M}^{2m+1}(\phi, \xi, \eta, g)$ is said to be Kenmotsu manifold if the following conditions hold \cite{KEN}:
\begin{equation}\label{2.4}
  \bar{\nabla}_X \xi=X-\eta(X)\xi,
\end{equation}
\begin{equation}\label{2.5}
 (\bar{\nabla}_X \phi)(Y)=g(\phi X ,Y)\xi-\eta(Y)\phi X,
\end{equation}
where $\bar{\nabla}$ denotes the Riemannian connection of $g$.\\
\indent Let $M$ be an $n$-dimensional submanifold of a Kenmotsu manifold $\bar{M}$. Throughout the paper we assume that
the submanifold $M$ of $\bar{M}$ is tangent to the structure vector field $\xi$. Let $\nabla$ and $\nabla ^\bot$ be the induced connections on the tangent bundle $TM$ and the normal bundle $T^\bot M$ of $M$ respectively.
Then the Gauss and Weingarten formulae are given by
\begin{equation}\label{2.6}
  \bar{\nabla}_X Y=\nabla_X Y+h(X,Y)
\end{equation}
and
\begin{equation}\label{2.7}
  \bar{\nabla}_XN=-A_N X+\nabla _X^ {\bot}N
\end{equation}
for all $X,Y\in \Gamma(TM)$ and $N\in \Gamma(T^\bot M)$, where $h$ and $A_N$ are second fundamental form and the shape operator
(corresponding to the normal vector field $N$) respectively for the immersion of $M$ into $\bar{M}$ and they are related by $g(h(X,Y),N)=g(A_N X,Y)$
for any $X,Y\in \Gamma(TM)$ and $N\in \Gamma(T^\bot M)$, where g is the Riemannian metric on $\bar{M}$ as well as on $M$.

The mean curvature $H$ of $M$ is given by $H=\frac{1}{n}\text{trace}\ h$. A submanifold $M$ of a Kenmotsu manifold $\bar{M}$ is said to be totally umbilical if $h(X,Y)=g(X,Y)H$ for any $X,Y\in \Gamma(TM)$. If $h(X,Y)=0$ for all $X,Y\in \Gamma(TM)$, then $M$ is totally geodesic and if $H=0$ then $M$ is minimal in $\bar{M}$.\\
\indent Let $\{e_1,\cdots,e_n\}$ be an orthonormal basis of the tangent bundle $TM$ and $\{e_{n+1},\cdots,e_{2m+1}\}$ be that of the normal bundle
 $T^\bot M$. Set
 \begin{equation}\label{2.8}
 h_{ij}^r=g(h(e_i,e_j),e_r)\ \text{and} \ \ \|h\|^2=g(h(e_i,e_j),h(e_i,e_j)),
 \end{equation}
 for $i,j\in \{1,\cdots, n\}$ and $r\in\{n+1,\cdots,2m+1\}$. For a differentiable function $f$ on $M$, the gradient $\boldsymbol{\nabla}f$ is defined by
 \begin{equation}\label{2.9}
 g(\boldsymbol{\nabla}f,X)=Xf
 \end{equation}
 for any $X\in\Gamma(TM)$. As a consequence, we get
 \begin{equation}\label{2.10}
 \|\boldsymbol{\nabla} f\|^2=\sum_{i=1}^{n}(e_i(f))^2.
 \end{equation}
\indent For any $X\in \Gamma(TM)$ and $N\in \Gamma(T^\bot M)$, we can write
\begin{eqnarray}
\label{2.11}
\text{(a)} \  \phi X = PX+QX,\ \  \text{(b)}\  \phi N= bN+cN
\end{eqnarray}
where $P X,\ bN$ are the tangential components and $QX,\ cN$ are the normal components.\\
\indent A submanifold $M$ of an almost contact metric manifold $\bar{M}$ is said to be invariant if $\phi(T_pM)\subseteq T_pM$  and anti-invariant if $\phi (T_pM)\subseteq T^\bot_pM$ for every $p\in M$.\\
\indent A submanifold $M$ of an almost contact metric manifold $\bar{M}$ is said to be slant if for each non-zero vector $X\in T_pM$, the angle $\theta$ between $\phi X$
and $T_pM$ is a constant, i.e. it does not depend on the choice of $p\in M$. Invariant and anti-invariant submanifolds are particular cases of slant submanifolds with slant angles $\theta=0$ and $\frac{\pi}{2}$ respectively.
\begin{theorem}\cite{CAR}
Let $M$ be a submanifold of an almost contact metric manifold $\bar{M}$ such that $\xi\in \Gamma(TM)$. Then, $M$ is slant if and
 only if there exists a constant $\lambda\in[0,1]$ such that
 \begin{equation}\label{2.12}
 P^2=\lambda(-I+\eta\otimes\xi),
 \end{equation}
 furthermore if $\theta$ is slant angle then $\lambda=\cos^2\theta$.
\end{theorem}
If $M$ is a slant submanifold of an almost contact metric manifold $\bar{M}$, the following relation holds \cite{UAS}:
\begin{equation}\label{2.15}
bQX=\sin^2\theta\{-X+\eta(X)\xi\}, \ \ \ cQX=-QPX.
\end{equation}
\begin{definition}\cite{BISHOP} Let $(N_1,g_1)$ and $(N_2,g_2)$ be two Riemannian manifolds with Riemannian metric $g_1$
and $g_2$ respectively and $f$ be a positive definite smooth function on $N_1$. The warped product
of $N_1$ and $N_2$ is the Riemannian manifold $N_1\times_{f}N_2 = (N_1\times N_2,g)$, where
\begin{equation}
\label{2.13}
g=g_1+f^2g_2.
\end{equation}
\end{definition}
\noindent A warped product manifold $N_1\times_{f}N_2$ is said to be trivial if the warping function $f$ is constant. For a warped product manifold $M=N_1\times_{f}N_2$, we have \cite{BISHOP}
\begin{eqnarray}\label{2.14}
\nabla_UX = \nabla_XU = (X\ln f) U
\end{eqnarray}
for any $X$, $Y\in\Gamma(TN_1)$ and $U\in\Gamma(TN_2)$.

We now recall the following:
\begin{theorem}\emph{(Hiepko's Theorem, see \cite{HIPKO})}.
Let $\mathcal{D}_1$ and $\mathcal{D}_2$ be two orthogonal distribution on a Riemannian manifold $M$. Suppose that $\mathcal{D}_1$ and $\mathcal{D}_2$ both are involutive such that $\mathcal{D}_1$ is a totally geodesic foliation and $\mathcal{D}_2$ is a spherical foliation. Then $M$ is locally isometric to a non-trivial warped product $M_2\times_fM_2$, where $M_2$ and $M_2$ are integral manifolds of $\mathcal{D}_1$ and $\mathcal{D}_2$, respectively.
\end{theorem}
\section{Skew $CR$-submanifolds of Kenmotsu manifolds}
Let $M$ be a submanifold of a Kenmotsu manifold $\bar{M}$. First from \cite{RONSSE}, we recall the definition of skew $CR$-submanifolds. Throughout the paper we consider the structure vector field $\xi$ is tangent to the submanifold otherwise the submanifold is $C$-totally real \cite{HASE}.\\
For any $X$ and $Y$ in $T_pM$, we have $g(PX,Y)=-g(X,PY)$. Hence it follows that $P^2$ is symmetric operator on the tangent space $TM$, for all $p\in M$. Therefore the eigen values are real and it is diagonalizable. Moreover its eigen values are bounded by $-1$ and $0$. For each $p\in M$, we may set
\begin{equation*}
\mathcal{D}^{\lambda}_{p}=ker\{P^2+\lambda^2(p)I\}_p,
\end{equation*}
where $I$ is the identity transformation and $\lambda(p)\in [0,1]$ such that $\-\lambda^2(p)$ is an eigen value of $P^2_p$. We note that $\mathcal{D}^1_{p}=ker Q$ and $\mathcal{D}^0_{p}=ker P$. $\mathcal{D}^1_{p}$ is the maximal $\phi$-invariant subspace of $T_pM$ and $\mathcal{D}^0_p$ is the maximal $\phi$-anti-invariant subspace of $T_pM$. From now on, we denote the distributions $\mathcal{D}^1$ and $\mathcal{D}^0$ by $\mathcal{D}^T\oplus <\xi>$ and $\mathcal{D}^\bot$, respectively. Since $P^2_p$ is symmetric and diagonalizable, if $-\lambda^2_1(p), \cdots, -\lambda^2_k(p)$ are the eigenvalues of $P^2$ at $p\in M$, then $T_pM$ can be decomposed as direct sum of mutually orthogonal eigen spaces, i.e.
\begin{equation*}
T_pM=\mathcal{D}^{\lambda_1}_p\oplus\mathcal{D}^{\lambda_2}_p \cdots \oplus\mathcal{D}^{\lambda_k}_p.
\end{equation*}
Each $\mathcal{D}^{\lambda_i}_p$, $1\leq i\leq k$ defined on $M$ with values in $(0,1)$ such that\\
(i) Each $-\lambda^2_i(p)$, $1\leq i\leq k$ is a distinct eigen value  of $P^2$ with
\begin{equation*}
T_pM=\mathcal{D}^T_p\oplus\mathcal{D}^\bot_p\oplus\mathcal{D}^{\lambda_1}_p\oplus\mathcal{D}^{\lambda_2}_p \cdots\oplus\mathcal{D}^{\lambda_k}_p\oplus <\xi>_p
\end{equation*}
for any $p\in M$.\\
(ii) The dimensions of  $\mathcal{D}^T_p,$ $\mathcal{D}_p^\bot$ and $\mathcal{D}^{\lambda_i}, 1\leq i\leq k$ are independent on $p\in M$. \\
Moreover, if each $\lambda_i$ is constant on $M$, then $M$ is called a skew CR-submanifold. Thus, we observe that  CR-submanifolds are a particular class of skew CR-submanifolds with $k=0$, $\mathcal{D}^T\neq \{0\}$ and $\mathcal{D}^\bot\neq \{0\}$. And slant submanifolds are also a particular class of skew CR-submanifold  with $k=1$, $\mathcal{D}^T=\{0\}$,  $\mathcal{D}^\bot= \{0\}$ and $\lambda_1$ is constant. Moreover, if $\mathcal{D}^\bot=\{0\}$, $\mathcal{D}^T\neq 0$ and $k=1$, then $M$ is semi-slant submanifold. Furthermore, if $\mathcal{D}^T=\{0\}$, $\mathcal{D}^\bot\neq \{0\}$ and $k=1$, then $M$ is a pseudo-slant (or hemi-slant) submanifold.\\
\indent A submanifold $M$ of $\bar{M}$ is said to be proper skew CR-submanifold of order 1 if $M$ is a skew CR-submanifold with $k=1$ and $\lambda_1$ is constant. In that case, the tangent bundle of $M$ is decomposed as
\begin{equation*}
TM=\mathcal{D}^T\oplus\mathcal{D}^\bot\oplus\mathcal{D}^\theta\oplus<\xi>.
\end{equation*}
The normal bundle $T^\bot M$ of a skew CR-submanifold $M$ is decomposed as
\begin{equation*}
T^\bot M=\phi\mathcal{D}^\bot\oplus Q\mathcal{D}^\theta\oplus \nu,
\end{equation*}
where $\nu$ is a $\phi$-invariant normal subbundle of $T^\bot M$.\\
\indent Now for the sake of further study we give the following useful results.
\begin{lemma}
Let $M$ be a proper skew $CR$-submanifold of order $1$ of a Kenmotsu manifold $\bar{M}$ such that $\xi \in\Gamma (\mathcal{D}^\bot\oplus \mathcal{D}^\theta)$, then we have
\begin{equation}\label{3.1}
 g(\nabla_XY,Z)=g(A_{\phi Z}X,\phi Y)-\eta(Z)g(X,Y),
\end{equation}
\begin{equation}\label{3.2}
 g(\nabla_XY,U)=\csc^2\theta [g(A_{QU}X,\phi Y)-g(A_{QPU}X,Y)]-\eta(U)g(X,Y)
\end{equation}
for every $X,Y \in \Gamma(\mathcal{D}^T)$, $Z\in \Gamma(\mathcal{D}^\bot)$ and $U\in \Gamma(\mathcal{D}^\theta)$.
\end{lemma}
\begin{proof}
For any $X,Y \in \Gamma(\mathcal{D}^T)$, $Z\in \Gamma(\mathcal{D}^\bot)$, we have
\begin{eqnarray*}
  g(\nabla_XY,Z) &=& g(\phi \bar{\nabla}_XY,\phi Z)+\eta(Z)g(\bar{\nabla}_XY,\xi) \\
  \nonumber &=& g(\bar{\nabla}_X\phi Y,\phi Z)-g((\bar{\nabla}_X\phi)Y,\phi Z)-\eta(Z)g(Y,\bar{\nabla}_X\xi).
\end{eqnarray*}
Using (2.4), (2.5) and (2.6) in the above equation, we get (\ref{3.1}).
Also, for $X,Y \in \Gamma(\mathcal{D}^T)$, $U\in \Gamma(\mathcal{D}^\theta)$, we have
\begin{eqnarray*}
  g(\nabla_XY,U) &=& g(\phi \bar{\nabla}_XY,\phi U)+\eta(U)g(\bar{\nabla}_XY,\xi) \\
  \nonumber &=& g(\bar{\nabla}_X\phi Y,\phi U)-g((\bar{\nabla}_X\phi)Y,\phi U)-\eta(U)g(Y,\bar{\nabla}_X\xi)\\
  \nonumber &=& g(\bar{\nabla}_X\phi Y,PU)+g(\bar{\nabla}_X\phi Y,QU)-\eta(U)g(X,Y)\\
  \nonumber &=& -g(\phi Y,\bar{\nabla}_XPU)-g(\phi Y,\bar{\nabla}_XQU)-\eta(U)g(X,Y)\\
  \nonumber &=& g(\bar{\nabla}_X\phi PU,Y)-g((\bar{\nabla}_X\phi)PU,Y)-g(\bar{\nabla}_XQU,\phi Y)-\eta(U)g(X,Y)\\
  \nonumber &=& g(\bar{\nabla}_X P^2U,Y)+g(\bar{\nabla}_XQPU,Y)-g(\bar{\nabla}_XQU,\phi Y)-\eta(U)g(X,Y).
\end{eqnarray*}
By virtue of (\ref{2.4}), (\ref{2.7}) and (\ref{2.12}) the above equation yields
\begin{eqnarray*}
  g(\nabla_XY,U) &=& -\cos^2\theta g(\bar{\nabla}_XU,Y)+\cos^2\theta\eta(U)g(X,Y) \\
  \nonumber &-& g(A_{QPU}X,Y)+g(A_{QU}X,\phi Y)- \eta(U)g(X,Y),
\end{eqnarray*}
Thus we get
\begin{eqnarray*}
\sin^2\theta g(\nabla_XY,U)=g(A_{QU}X,\phi Y)-g(A_{QPU}X,Y)-\sin^2\theta\eta(U)g(X,Y).
\end{eqnarray*}
From which the relation (3.2) follows.
\end{proof}
\begin{corollary}
Let $M$ be a proper skew $CR$-submanifold of order $1$ of a Kenmotsu manifold $\bar{M}$ such that $\xi\in\Gamma(\mathcal{D}^\bot\oplus \mathcal{D}^\theta)$, then we have
\begin{equation}\label{3.3}
g([X,Y],Z)=g(A_{\phi Z}X,\phi Y)-g(A_{\phi Z}Y,\phi X)
\end{equation}
\begin{equation}\label{3.4}
g([X,Y],U)=\csc^2\theta \{g(A_{QU}X,\phi Y)-g(A_{QU}Y,\phi X)\}
\end{equation}
for every $X,Y \in \Gamma(\mathcal{D}^T)$, $Z\in \Gamma(\mathcal{D}^\bot)$ and $U\in\Gamma(\mathcal{D}^\theta)$.
\end{corollary}
\begin{lemma}
Let $M$ be a proper skew $CR$-submanifold of order $1$ of a Kenmotsu manifold $\bar{M}$ such that $\xi\in\Gamma(\mathcal{D}^\bot\oplus \mathcal{D}^\theta)$, then we have
\begin{equation}\label{3.5}
 g(\nabla_ZW,X)=-g(A_{\phi W}\phi X,Z),
\end{equation}
\begin{equation}\label{3.6}
 g(\nabla_ZU,X)=\csc^2\theta\{g(A_{QPU}X,Z)-g(A_{QU}\phi X,Z)\},
\end{equation}
\begin{equation}\label{3.7}
g(\nabla_UZ,X)=-g(A_{\phi Z}X,U),
\end{equation}
\begin{equation}\label{3.8}
g(\nabla_UV,X)=\csc^2\theta\{g(A_{QPV}X,U)-g(A_{QV}\phi X,U)\},
\end{equation}
\begin{equation}\label{3.9}
g(\nabla_XZ,U)=\sec^2\theta\{g(A_{QPU}Z,X)-g(A_{\phi Z}PU,X)
\end{equation}
 for $X\in \Gamma(\mathcal{D}^T)$, $Z,W\in \Gamma(\mathcal{D}^\bot)$ and $U,V\in\Gamma(\mathcal{D}^\theta)$.
\end{lemma}
\begin{proof}
For every $X\in \Gamma(\mathcal{D}^T)$ and $Z,W\in \Gamma(\mathcal{D}^\bot)$, we have
\begin{eqnarray*}
  g(\nabla_ZW,X) &=& g(\phi \bar{\nabla}_ZW,\phi X), \\
  \nonumber &=& g(\bar{\nabla}_Z\phi W,\phi X)-g((\bar{\nabla}_Z\phi) W,\phi X).
\end{eqnarray*}
  Using (\ref{2.5}), (\ref{2.7}) and orthogonality of vector fields in the above equation, we get (\ref{3.5}). Also, for $X\in\Gamma(\mathcal{D}^T)$, $Z\in\Gamma(\mathcal{D}^\bot)$ and $U\in\Gamma(\mathcal{D}^\theta)$, we have
\begin{eqnarray*}
  g(\nabla_ZU,X) &=& g(\phi\bar{\nabla}_ZU,\phi X), \\
  \nonumber &=& g(\bar{\nabla}_Z\phi U,\phi X)-g((\bar{\nabla}_Z\phi)U,\phi X),  \\
  \nonumber &=& g(\bar{\nabla}_Z PU,\phi X)+g(\bar{\nabla}_ZQU,\phi X), \\
  \nonumber &=& -g(\bar{\nabla}_Z P^2U,X)-g(\bar{\nabla}_ZQPU,X)+g(\bar{\nabla}_ZQU,\phi X).
\end{eqnarray*}
Using (\ref{2.7}), (\ref{2.12}) and the symmetry of shape operator in the above equation, we obtain
\begin{equation*}
g(\nabla_ZU,X)= \cos^2\theta g(\bar{\nabla}_ZU,X)+ g(A_{QPU}X,Z)-g(A_{QU}\phi X,Z),
\end{equation*}
from which the relation (\ref{3.6}) follows.\\
Again, for $X\in \Gamma(\mathcal{D}^T)$, $Z\in \Gamma(\mathcal{D}^\bot)$ and $U\in \Gamma(\mathcal{D}^\theta)$,  we have
\begin{equation*}
 g(\nabla_UZ,X)=g(\phi \bar{\nabla}_UZ,\phi X)=g(\bar{\nabla}_U\phi Z,\phi X)-g((\bar{\nabla}_U\phi)Z,X).
\end{equation*}
By virtue of (\ref{2.5}) and (\ref{2.7}), the above equation yields
\begin{equation*}
  g(\nabla_UZ,X)=-g(A_{\phi Z}U,\phi X)-g(\phi U,X)\eta(Z),
\end{equation*}
from which the relation (\ref{3.7}) follows. Again we have
\begin{eqnarray*}
  g(\nabla_UV,X) &=& g(\phi \bar{\nabla}_UV,\phi X), \\
  \nonumber &=& g(\bar{\nabla}_U\phi V,\phi X)- g((\bar{\nabla}_U\phi) V,\phi X),\\
  \nonumber &=& g(\bar{\nabla}_UPV,\phi X)+ g(\bar{\nabla}_UQV,\phi X), \\
  \nonumber &=& -g(\bar{\nabla}_UP^2V,X)- g(\bar{\nabla}_UQPV,X)+g(\bar{\nabla}_UQV,\phi X).
\end{eqnarray*}
Using (\ref{2.7}), (\ref{2.12}) and the symmetry of shape operator in the above equation, we get
\begin{equation*}
  g(\nabla_UV,X)=\cos^2\theta g(\bar{\nabla}_UV,X)+g(A_{QPV}X,U)-g(A_{QV}\phi X,U),
\end{equation*}
from which we get (\ref{3.8}).\\
For every $X\in \Gamma(\mathcal{D}^T)$, $Z\in \Gamma(\mathcal{D}^\bot)$ and $U\in\Gamma(\mathcal{D}^\theta)$, we have
\begin{eqnarray*}
  g(\nabla_XZ,U) &=& g(\phi \bar{\nabla}_XZ,\phi U)+\eta(U)g(\bar{\nabla}_XZ,\xi), \\
  \nonumber &=& g(\bar{\nabla}_X\phi Z,\phi U-g((\bar{\nabla}_X\phi)Z,\phi U)-\eta(U)g(Z,\bar{\nabla}_X\xi), \\
  \nonumber &=&  g(\bar{\nabla}_X\phi Z,PU)+g(\bar{\nabla}_X\phi Z,QU)+\eta(Z)g(\phi X,\phi U)-\eta(U)g(X,Z), \\
  \nonumber &=& g(\bar{\nabla}_X\phi Z,PU)-g(\bar{\nabla}_XQU,\phi Z), \\
  \nonumber &=& g(\bar{\nabla}_X\phi Z,PU)+g(\bar{\nabla}_XbQU,Z)+g(\bar{\nabla}_XcQU,Z).
\end{eqnarray*}
In view of (\ref{2.7}), (\ref{2.15}) and the symmetry of shape operator, the above equation reduces to
\begin{equation*}
  g(\nabla_XZ,U)=-g(A_{\phi Z}PU,X)-\sin^2\theta g(\bar{\nabla}_XU,Z)+g(A_{QPU}Z,X),
\end{equation*}
from which the relation (\ref{3.9}) follows.
\end{proof}
\section{Warped product skew CR-submanifolds of the form $M_2\times_fM_T$}
Let $M=M_2\times_fM_T$ be a warped product skew $CR$-submanifold of order $1$ of a Kenmotsu manifold $\bar{M}$ such that $\xi$ is tangent to $M_2=M_\perp\times M_\theta$, where $M_T$, $M_\theta$ and $M_\perp$ are invariant, proper slant and anti-invariant submanifold of $\bar{M}$, respectively. Let the dimensions of these submanifolds are $dim \ M_\bot=d_1$, $dim \ M_\theta=d_2$ and $dim \ M_T=d_3$. If $d_2=0$ then $M$ is a $CR$-warped product of the form $M=M_\bot\times_fM_T$ which have been studied in \cite{UAN}. \\
\indent Now, we construct an example of a non-trivial warped product skew $CR$-submanifold of order $1$ of the form $M=M_2\times_fM_T$.
\begin{example}
Consider the Kenmotsu manifold $M=\mathbb{R}\times_fC^4$ with the structure $(\phi, \xi,\eta,g)$ is given by
\begin{equation*}
\phi \bigg(\sum_{i=1}^{5}(X_i\frac{\partial}{\partial x^i}+Y_i\frac{\partial}{\partial y^i})+Z\frac{\partial}{\partial t}\bigg)=Y_i\frac{\partial}{\partial x^i}-X_i\frac{\partial}{\partial y^i},
\end{equation*}
$\xi=3e^{-t}\frac{\partial}{\partial t}$, $\eta=\frac{1}{3}e^t dt$ and $g=\eta\otimes\eta+\frac{e^{3t}}{9}\displaystyle\sum_{i=1}^{5}(dx^i\otimes dx^i+dy^i\otimes dy^i)$
Now, we consider a submanifold $M$ of $\bar{M}$ defined by the immersion $\chi$ as follows:
\begin{equation*}
\chi(u,v,w,s,\theta,\phi,t)=3(e^{-t}u,0,w,0,2\theta+3\phi,0,e^{-t}v,s,0,3\theta+2\phi,t).
\end{equation*}
Then the local orthonormal frame of $TM$ is spanned by the following:
\begin{equation*}
Z_1=\frac{3}{e^t}(\frac{\partial}{\partial x^1}),\quad Z_2=\frac{3}{e^t}(\frac{\partial}{\partial y^2}),\quad Z_3=3\frac{\partial}{\partial x^3},\quad Z_4=3\frac{\partial}{\partial y^3}
\end{equation*}
\begin{equation*}
Z_5=3(2\frac{\partial}{\partial x^5}+3\frac{\partial}{\partial y^5}),\quad Z_6=3(3\frac{\partial}{\partial x^5}+2\frac{\partial}{\partial y^5}),\quad Z_7=3\frac{\partial}{\partial t}.
\end{equation*}
Also, we have
\begin{equation*}
\phi Z_1=-\frac{3}{e^t}(\frac{\partial}{\partial y^1}), \phi Z_2=\frac{3}{e^t}(\frac{\partial}{\partial x^2}),\quad \phi Z_3=-3\frac{\partial}{\partial y^3}, \phi Z_4=3\frac{\partial}{\partial x^3},
\end{equation*}
\begin{equation*}
\phi Z_5=3(-2\frac{\partial}{\partial y^5}+3\frac{\partial}{\partial x^5}),\quad \phi Z_6=3(-3\frac{\partial}{\partial y^5}+2\frac{\partial}{\partial x^5}), \quad \phi Z_7=0.
\end{equation*}
If we define $D^\bot=span\{Z_1,Z_2,Z_7\}$, $D^\theta=span\{Z_5,Z_6\}$ and $D^T=span\{Z_3,Z_4\}$ then by simple calculations we can say that $D^T$ is an  invariant distribution and $D^\theta$ is a slant distribution with slant angle $\cos^{-1}\frac{5}{13}$. Hence $M$ is a proper skew $CR$-submanifold of $\bar{M}$ of order $1$. Also, it is clear that $D^\bot\oplus D^\theta$ and $D^T$ both are integrable. If we denote the integral manifolds of $D^\bot\oplus D^\theta$ and $D^T$ by $M_2$ and $M_T$ respectively, then the metric tensor $g_M$ of $M$ is given by
\begin{eqnarray*}
  g_M &=& (dw^2+ds^2)+13(d\theta^2+d\phi^2)+e^{3t}(dw^2+ds^2) \\
  \nonumber &=& g_{M_2}+e^{3t}(dw^2+ds^2).
\end{eqnarray*}
Thus $M=M_2\times_fM_T$ is a warped product skew $CR$-submanifold of $\bar{M}$ with the warping function $f=\sqrt{e^{3t}}$.
\end{example}
Now, we prove the followings:
\begin{lemma}
Let $M=M_2\times_fM_T$ be a warped product skew $CR$-submanifold of order $1$ of a Kenmotsu manifold $\bar{M}$ such that $\xi$ is tangent to $M_2=M_\perp\times M_\theta$, then we have
\begin{equation}\label{4.3}
\xi \ln f=1,
\end{equation}
\begin{equation}\label{4.4}
g(h(X,Z),\phi W)=0,
\end{equation}
\begin{equation}\label{4.5}
g(h(X,U),\phi Z)=g(h(X,Z),QU)=0,
\end{equation}
and
\begin{equation}\label{4.6}
g(h(X,U),QV)=0
\end{equation}
for every $X\in \Gamma(M_T)$, $Z,W\in\Gamma(M_\perp)$ and $U,V \in \Gamma(M_\theta)$.
\end{lemma}
\begin{proof}
The proof of (\ref{4.3}) is similar as in \cite{NAGHI}.\\
Now, for $X\in \Gamma(M_T)$ and $Z\in \Gamma(M_\perp)$, we have
\begin{eqnarray*}
g(h(X,Z),\phi W)&=&g(\bar{\nabla}_ZX,\phi W) \\
  \nonumber &=&-g(\bar{\nabla}_Z\phi X,W)+g((\bar{\nabla}_Z\phi)X,W).
\end{eqnarray*}
Using (\ref{2.5}) and (\ref{2.14}) in the above equation, we obtain
\begin{equation}\label{4.8}
g(h(X,Z),\phi W)=-(Z\ln f)g(\phi X,W)=0.
\end{equation}
Thus, we get (\ref{4.4}).
Again, for $X\in \Gamma(M_T)$,  $Z\in \Gamma(M_\perp)$ and $U\in \Gamma(M_\theta)$, we have
\begin{equation*}
g(h(X,U),\phi Z)=g(\bar{\nabla}_UX,\phi Z)=-g(\bar{\nabla}_U\phi X,Z)+g((\bar{\nabla}_U\phi)X,Z).
\end{equation*}
Using (\ref{2.5}) and (\ref{2.14}), the above equation reduces to
\begin{equation}\label{4.9}
g(h(X,U),\phi Z)=-(U\ln f)g(\phi X,Z)=0.
\end{equation}
Also,
\begin{eqnarray*}
 g(h(X,U),\phi Z) &=& g(\bar{\nabla}_XU,\phi Z), \\
  \nonumber &=& -g(\bar{\nabla}_X\phi U,Z)+g((\bar{\nabla}_X\phi)U,Z), \\
  \nonumber &=&-g(\bar{\nabla}_XPU,Z)-g(\bar{\nabla}_XQU,Z)+g((\bar{\nabla}_X\phi)U,Z).
\end{eqnarray*}
Using (\ref{2.5}), (\ref{2.7}) and (\ref{2.14}) in the above equation, we obtain
\begin{equation}\label{4.10}
g(h(X,U),\phi Z)=g(h(X,Z),QU).
\end{equation}
From (\ref{4.9}) and (\ref{4.10}) we get (\ref{4.5}).
Again, for $X\in \Gamma(M_T)$ and $U,V\in \Gamma(M_\theta)$ we have
\begin{eqnarray*}
g(h(X,U),QV) &=& g(\bar{\nabla}_UX,QV) \\
  \nonumber &=& g(\bar{\nabla}_UX,\phi V)-g(\bar{\nabla}_UX,PV) \\
  \nonumber &=&-g(\bar{\nabla}_U\phi X,V)+g((\bar{\nabla}_U\phi)X,V)-g(\bar{\nabla}_UX,PV).
\end{eqnarray*}
By virtue of (\ref{2.5}) and (\ref{2.14}), the above equation yields
\begin{eqnarray*}
  g(h(X,U),QV) &=& -(U\ln f)g(\phi X,V)+\eta(V)g(\phi U,X)  -(U\ln f)g(X,PV)\\
&=&0.
\end{eqnarray*}
Thus we get (\ref{4.6}).
\end{proof}
\begin{proposition}
Let $M=M_2\times_fM_T$ be a warped product skew CR-submanifold of order 1 of a Kenmotsu manifold $\bar{M}$ such that $\xi$ is tangent
 to $M_2=M_\perp\times M_\theta$, then we have $h(X,E)\in \nu$ for every $X\in\Gamma(M_T)$ and $E\in\Gamma(M_2)$
\end{proposition}
\begin{proof}
The proof is obvious from (\ref{4.4}), (\ref{4.5}),  (\ref{4.6}) and the fact that $h(X,\xi)=0$, for every $X\in \Gamma(M_T)$.
\end{proof}
\begin{lemma}
  Let $M=M_2\times_fM_T$ be a warped product skew CR-submanifold of order 1 of a Kenmotsu manifold $\bar{M}$ such that $\xi$ is tangent to
  $M_2=M_\perp\times M_\theta$, then we have
  \begin{equation}\label{4.11}
  g(h(X,Y),\phi Z)=\{(Z\ln f)-\eta(Z)\}g(X,\phi Y),
  \end{equation}
\begin{equation}\label{4.12}
g(h(X,Y),QU)=\{\eta(U)-(U\ln f)\}g(\phi X,Y)+(PU\ln f)g(X,Y)
\end{equation}
and
\begin{equation}\label{4.13}
g(h(X,Y),QPU)=\cos^2\theta\{\eta(U)-(U\ln f)\}g(X,Y)-(PU\ln f)g(\phi X,Y),
\end{equation}
for every $X,Y\in \Gamma(M_T)$, $Z\in \Gamma(M_\perp)$ and $U\in \Gamma(M_\theta)$.
\end{lemma}
\begin{proof}
  For every $X,Y\in \Gamma(M_T)$ and $Z\in \Gamma(M_\perp)$, we have
  \begin{eqnarray*}
 g(h(X,Y),\phi Z) &=&g(\bar{\nabla}_XY,\phi Z) \\
 &=& -g(\bar{\nabla}_X\phi Y,Z)+g((\bar{\nabla}_X\phi)Y,Z) \\
 &=& g(\phi Y,\bar{\nabla}_XZ)+\eta(Z)g(\phi X,Y).
 \end{eqnarray*}
Using (\ref{2.14}) in the above equation, we obtain
\begin{equation*}
 g(h(X,Y),\phi Z)=(Z\ln f)g(X,\phi Y)+\eta(Z)g(\phi X,Y),
\end{equation*}
from which the relation (\ref{4.11}) follows.

Also, for every  $X,Y\in \Gamma(M_T)$ and $U\in \Gamma(M_\theta)$, we have
\begin{eqnarray*}
g(h(X,Y),QU) &=& g(\bar{\nabla}_XY,\phi U)-g(\bar{\nabla}_XY,PU) \\
   &=&-g(\bar{\nabla}_X\phi Y,U)+g((\bar{\nabla}_X\phi)Y,U)+g(\bar{\nabla}_XPU,Y).
\end{eqnarray*}
Using (\ref{2.5}) and (\ref{2.14}) in the above equation, we obtain
\begin{equation*}
g(h(X,Y),QU)=(U\ln f)g(X,\phi Y)+\eta(U)g(\phi X,Y)+(PU\ln f)g(X,Y),
\end{equation*}
from which the relation (\ref{4.12}) follows.
Also, replacing $U$ by $PU$ in (\ref{4.12}) and using (\ref{2.12}), we get (\ref{4.13}).
\end{proof}

Now, replacing $X$ by $\phi X$ and $Y$ by $\phi Y$ in (\ref{4.11}), we obtain the following:
\begin{equation}\label{4.14}
g(h(\phi X,Y),\phi Z)=\{(Z\ln f)-\eta(Z)\}g(X,Y),
\end{equation}
\begin{equation}\label{4.15}
g(h(X,\phi Y),\phi Z)=\{\eta(Z)-(Z\ln f)\}g(X,Y)
\end{equation}
and
\begin{equation}\label{4.16}
g(h(\phi X,\phi Y),\phi Z)=\{(Z\ln f)-\eta(Z)\}g(X,\phi Y).
\end{equation}
Also, replacing $X$ by $\phi X$ and $Y$ by $\phi Y$ in (\ref{4.12}), we get the following:
\begin{equation}\label{4.17}
g(h(\phi X,Y),QU)=\{\eta(U)-(U\ln f)\}g(X,Y)+(PU\ln f)g(\phi X,Y).
\end{equation}
\begin{equation}\label{4.18}
g(h(X,\phi Y),QU)=-\{\eta(U)-(U\ln f)\}g(X,Y)-(PU\ln f)g(\phi X,Y)
\end{equation}
and
\begin{equation}\label{4.19}
g(h(\phi X,\phi Y),QU)=\{\eta(U)-(U\ln f)\}g(\phi X,Y)+(PU\ln f)g(X,Y).
\end{equation}
\begin{corollary}
Let $M=M_2\times_fM_T$ be a warped product skew CR-submanifold of order 1 of a Kenmotsu manifold $\bar{M}$ such that $\xi$ is tangent to $M_2$ and
  $M_2=M_\perp\times M_\theta$, then we have\\
(i) $g(h(\phi X,Y),\phi Z)=-g(h(X,\phi Y),\phi Z)$,\\ (ii) $g(h(\phi X,\phi Y),\phi Z)=g(h(X,Y),\phi Z)$,\\
 (iii) $g(h(\phi X,Y),QU)=-g(h(X,\phi Y), QU)$,\\ and (iv) $g(h(\phi X,\phi Y),QU)=g(h(X,Y),QU)$.
\end{corollary}
\begin{proof}
  The relation (i) follows from (\ref{4.14}) and (\ref{4.15}). \\ The relation (ii) follows from (\ref{4.11}) and (\ref{4.16}).\\
  The relation (iii) follows from (\ref{4.17}) and (\ref{4.18}).\\The relation (iv) follows from (\ref{4.12}) and (\ref{4.19}).
\end{proof}

\maketitle
\section{Characterization of Skew CR-warped products of the form $M_2\times_f M_T $}

Now, we obtain a characterization for a  proper skew CR-warped product submanifold of order 1 of the form $M=M_2\times_fM_T$ such that $M_2=M_\perp\times M_\theta$ of a Kenmotsu manifold $\bar{M}$.
 \begin{theorem}
Let $M$ be a proper skew CR-submanifold of order 1 of a Kenmotsu manifold $\bar{M}$ such that $\xi$ is orthogonal to the invariant distribution $\mathcal{D}^T$, then $M$ is locally a warped product skew CR-submanifold if and only if
\begin{equation}\label{5.1}
A_{\phi Z}X=\{\eta(Z)-(Z\mu)\}\phi X,
\end{equation}
\begin{equation}\label{5.2a}
A_{QU}X=\{\eta(U)-(U\mu)\}\phi X+(PU\mu)X,
\end{equation}
and
\begin{equation}\label{5.3a}
(\xi\mu)=1
\end{equation}
for every $X\in\Gamma(\mathcal{D}^T)$,  $Z\in \Gamma(\mathcal{D}^\bot)$, $U\in \Gamma(\mathcal{D}^\theta)$ and for some smooth function $\mu$ on $M$ satisfying $Y(\mu)=0$, for any $Y\in \Gamma(\mathcal{D}^T)$.
\end{theorem}
\begin{proof}
Let $M=M_2\times_f M_T$ be a proper warped product skew CR-submanifold of order 1 of a Kenmotsu manifold $\bar{M}$ such that $M_2=M_\perp\times M_\theta$. We denote the tangent space of $M_T$, $M_\perp$ and $M_\theta$ by $\mathcal{D}^T$, $\mathcal{D}^\bot$ and $\mathcal{D}^\theta$, respectively. Then from (\ref{4.4}) and from (\ref{4.5}), we have
\begin{equation}\label{5.2}
A_{\phi Z}X\perp \mathcal{D}^\bot
\end{equation}
and
\begin{equation}\label{5.3}
A_{\phi Z}X\perp \mathcal{D}^\theta
\end{equation}
  for every $X\in \Gamma(\mathcal{D}^T)$ and $Z\in \Gamma(\mathcal{D}^\bot)$ respectively. Also since $h(B,\xi)=0$, for every $B\in \Gamma(TM)$, we have
\begin{equation}\label{5.4}
g(A_{\phi Z}X,\xi)=g(h(X,\xi),\phi Z)=0.
\end{equation}
 From (\ref{5.2}), (\ref{5.3}) and (\ref{5.4}), we can say that
 \begin{equation}\label{5.5}
A_{\phi Z}X\in \Gamma(\mathcal{D}^T).
 \end{equation}
 From (\ref{4.11}) and (\ref{5.5}), we get (\ref{5.1}). Also from (\ref{4.5}), we have
 \begin{equation}\label{5.6}
 A_{QU}X\perp \mathcal{D}^\bot,
\end{equation}
for every $X\in \Gamma(\mathcal{D}^T)$ and $U\in \Gamma(\mathcal{D}^\theta)$, and from (\ref{4.6}), we have
\begin{equation}\label{5.7}
A_{QU}X\perp\mathcal{D}^\theta
\end{equation}
for every $X\in \Gamma(\mathcal{D}^\bot)$ and $U\in \Gamma(\mathcal{D}^\theta)$.\\  From (\ref{5.4}), (\ref{5.6}) and (\ref{5.7}), we can say that
\begin{equation}\label{5.8}
A_{QU}X\in \Gamma(\mathcal{D}^T),
\end{equation}
for every $X\in \Gamma(\mathcal{D}^T)$ and $U\in \Gamma(\mathcal{D}^\theta)$. The relation (\ref{5.2a}) follows from
(\ref{4.12}) and (\ref{5.8}) and also (\ref{5.3a}) follows from (\ref{4.3}).\\
\indent Conversely, let $M$ be a proper skew CR-submanifold of order $1$ of a Kenmotsu manifold $\bar{M}$ such that (\ref{5.1})-(\ref{5.3a}) holds. Then from (\ref{3.5}), (\ref{3.7}) and (\ref{5.1}), we have
\begin{equation}\label{5.9}
g(\nabla_ZW,X)=0
\end{equation}
and
\begin{equation}\label{5.10}
g(\nabla_UZ,X)=0
\end{equation}
 for every $X\in \Gamma(\mathcal{D}^T)$ and $Z,W\in \Gamma(\mathcal{D}^\bot)$. Also, from (\ref{3.6}), (\ref{3.8}) and (\ref{5.2a}), we have
\begin{equation}\label{5.11}
g(\nabla_ZU,X)=0
\end{equation}
and
\begin{equation}\label{5.12}
g(\nabla_UV,X)=0
\end{equation}
 for every $X\in \Gamma(\mathcal{D}^T)$,  $Z\in \Gamma(\mathcal{D}^\bot)$ and $U,V\in \Gamma(\mathcal{D}^\theta)$.  Hence
from (\ref{5.9})-(\ref{5.12}), we can conclude that
\begin{equation*}
g(\nabla_EF,X)=0
\end{equation*}
for every $E,F\in \Gamma(\mathcal{D}^\bot\oplus\mathcal{D}^\theta+\{\xi\})$ and $X\in \Gamma(\mathcal{D}^T)$.\\ Therefore, the leaves of $\mathcal{D}^\bot\oplus\mathcal{D}^\theta+\{\xi\}$ are totally geodesic in $M$.\\
Now, from (\ref{3.3}) and (\ref{5.1}), we have
\begin{equation}\label{5.13}
g([X,Y],Z)=0
\end{equation}
for every $X,Y\in \Gamma(\mathcal{D}^T)$ and $Z\in \Gamma(\mathcal{D}^\bot)$,  Also from (\ref{3.4}) and (\ref{5.2a}), we have
\begin{equation}\label{5.14}
g([X,Y],U)=0
\end{equation}
for every $X,Y\in \Gamma(\mathcal{D}^T)$ and $U\in \Gamma(\mathcal{D}^\theta)$.\\
 Since $h(A,\xi)=0$ for every $A\in \Gamma(TM)$, we have from (\ref{5.13}) and (\ref{5.14}) that
\begin{equation*}
g([X,Y],E)=0
\end{equation*}
for every $X,Y\in \Gamma(\mathcal{D}^T)$ and $E\in \Gamma(\mathcal{D}^\bot\oplus\mathcal{D}^\theta+\{\xi\})$.
Consequently the distribution $\mathcal{D}^T$ is integrable.\\ Next, we consider the integrable manifold $M_T$ of $\mathcal{D}^T$ and let $h^T$ be the second fundamental form of $M_T$ in $M$. Then for any $X,Y\in \Gamma(\mathcal{D}^T)$, we have from (\ref{3.1}) that
\begin{eqnarray}\label{5.15}
  g(h^T(X,Y),Z) &=& g(\nabla_XY,Z) \\
  \nonumber &=& g(A_{\phi Z}X,\phi Y)-\eta(Z)g(X,Y).
\end{eqnarray}
By virtue of (\ref{5.1}), (\ref{5.15}) yields
\begin{equation}\label{5.16}
g(h^T(X,Y),Z)=-(Z\mu)g(X,Y).
\end{equation}
Similarly for any $X,Y\in \Gamma(\mathcal{D}^T)$ and $U\in \Gamma(\mathcal{D}^\theta)$, we have from (\ref{3.2}) that
\begin{eqnarray}\label{5.17}
  &&g(h^T(X,Y),U = g(\nabla_XY,U) \\
  \nonumber &=& \csc^2\theta[g(A_{QU}X,\phi Y)-g(A_{QPU}X,Y)]-\eta(U)g(X,Y).
\end{eqnarray}
In view of (\ref{5.2a}), (\ref{5.17}) reduces to
\begin{eqnarray}\nonumber
g(h^T(X,Y),U)&=& \csc^2\theta[\{\eta(U)-(U\mu)\}g(\phi X,\phi Y)+(PU\mu)g(X,\phi Y) \\
\nonumber &-& \cos^2\theta\{\eta(U)-(U\mu)\}g(X,Y)+(PU\mu)g(\phi X, Y)]-\eta(U)g(X,Y) \\
\label{5.18}&=& -(U\mu)g(X,Y).
\end{eqnarray}
Also for any $X,Y\in \Gamma(\mathcal{D}^T)$, we have
\begin{equation*}
g(h^T(X,Y),\xi)=g(\nabla_XY,\xi)= -g(Y,\bar{\nabla}_X\xi)=-g(X,Y).
\end{equation*}
Using (\ref{5.3a}) in the above equation we obtain
\begin{equation}\label{5.19}
g(h^T(X,Y),\xi)=-(\xi\mu)g(X,Y).
\end{equation}
From (\ref{5.16}), (\ref{5.18}) and (\ref{5.19}), we conclude that
\begin{equation*}
  g(h^T(X,Y),E)=-g(\boldsymbol{\nabla}\mu,E)g(X,Y)
\end{equation*}
for every $X,Y\in \Gamma(\mathcal{D}^T)$ and $E\in \Gamma(\mathcal{D}^\bot\oplus\mathcal{D}^\theta\oplus\{\xi\})$.
Consequently, $M_T$ is totally umbilical in $\bar{M}$ with mean curvature vector $H^T=-\boldsymbol{\nabla}\mu$.\\
Finally, we show that $H^T$ is parallel with respect to the normal connection $D^N$ of $M_T$ in $M$. We take $E\in \Gamma(\mathcal{D}^\bot\oplus\mathcal{D}^\theta\oplus\{\xi\})$ and $X\in \Gamma(\mathcal{D}^T)$, then we have
\begin{equation*}
g(D^N_X\boldsymbol{\nabla}\mu,E)=g(\nabla_X\boldsymbol{\nabla}^\bot\mu,Z)+g(\nabla_X\boldsymbol{\nabla}^\theta_\mu,U)
+g(\nabla_X\boldsymbol{\nabla}^\xi_\mu,\xi),
\end{equation*}
where $\boldsymbol{\nabla}^\bot\mu$, $\boldsymbol{\nabla}^\theta\mu$ and $\boldsymbol{\nabla}^\xi\mu$ are the gradient components of $\mu$
on $M$ along $\mathcal{D}^\bot$, $\mathcal{D}^\theta$ and $\{\xi\}$ respectively. Then by the property of Riemannian metric, the above equation reduces to
\begin{eqnarray*}
  g(D^N_X\boldsymbol{\nabla}\mu,E) &=& Xg(\boldsymbol{\nabla}^\bot\mu,Z)-g(\boldsymbol{\nabla}^\bot\mu,\nabla_XZ)+Xg(\boldsymbol{\nabla}^\theta\mu,U) \\
  \nonumber &&-g(\boldsymbol{\nabla}^\theta\mu,\nabla_XU)+Xg(\boldsymbol{\nabla}^\xi\mu,\xi)-g(\boldsymbol{\nabla}^\xi\mu,\nabla_X\xi) \\
  \nonumber &=& X(Z\mu)-g(\boldsymbol{\nabla}^\bot\mu,[X,Z])-g(\boldsymbol{\nabla}^\bot\mu,\nabla_ZX) \\
 \nonumber  && +X(U\mu)-g(\boldsymbol{\nabla}^\theta\mu,[X,U])-g(\boldsymbol{\nabla}^\theta\mu,\nabla_UX) \\
  \nonumber &&+X(\xi\mu)-g(\boldsymbol{\nabla}^\xi\mu,[X,\xi])-g(\boldsymbol{\nabla}^\xi\mu,\nabla_\xi X) \\
  \nonumber &=& Z(X\mu)+g(\nabla_Z\boldsymbol{\nabla}^\bot\mu,X)+U(X\mu)+g(\nabla_U\boldsymbol{\nabla}^\theta\mu,X) \\
  \nonumber && +\xi(X\mu)+g(\nabla_\xi\boldsymbol{\nabla}^\xi\mu,X) \\
  \nonumber &=& 0,
\end{eqnarray*}
since $(X\mu)=0$, for any $X\in \Gamma(\mathcal{D}^T)$ and $ \nabla_Z\boldsymbol{\nabla}^\bot\mu+\nabla_U\boldsymbol{\nabla}^\theta\mu+\nabla_\xi\boldsymbol{\nabla}^\xi\mu=\nabla_E\boldsymbol{\nabla}\mu$
is orthogonal to $\mathcal{D}^T$ for any $E\in \Gamma(\mathcal{D}^\bot\oplus\mathcal{D}^\theta\oplus\{\xi\})$ as $\boldsymbol{\nabla}\mu$
is the gradient along $M_2$ and $M_2$ is totally geodesic in $\bar{M}$. Therefore, the mean curvature vector $H^T$ of $M_T$ is parallel. Thus, $M_T$
is an extrinsic sphere in $M$. Hence by Theorem 2.2, $M$ is  locally a warped product submanifold. Thus the proof is complete.
\end{proof}
\begin{corollary}
  Let $M$ be a proper skew CR-submanifold of order $1$ of a Kenmotsu manifold $\bar{M}$ such that $\xi$ is tangent to the anti-invariant distribution
  $\mathcal{D}^\bot,$ then $M$ is locally a warped product submanifold if and only if\\
  (i)$A_{\phi Z}X=\{(Z\mu)-\eta(Z)\}\phi X$,\\ (ii) $A_{Q U}X=(PU\mu)X-(U\mu)\phi X$\\ (iii) $(\xi\mu)=1$,\\
  for every $X\in \Gamma(\mathcal{D}^T), \  Z\in \Gamma(\mathcal{D}^\bot), \ U\in\Gamma(\mathcal{D}^\theta) $ and for some smooth function $\mu$
  on $M$ satisfying $Y(\mu)=0$, for any $Y\in \Gamma(\mathcal{D}^T)$.
\end{corollary}
\begin{corollary}
  Let $M$ be a proper Skew CR-submanifold of order $1$ of a Kenmotsu manifold $\bar{M}$ such that $\xi$ is tangent to the slant distribution $\mathcal{D}^\theta$, then $M$ is locally a warped product submanifold if and only if \\ (i) $A_{\phi Z}X=(Z\mu)\phi X$,\\
   (ii) $A_{Q U}X=\{\eta(U)-(U\mu)\}\phi X+(PU\mu)X$\\ (iii) $(\xi\mu)=1$,\\ for every $X\in \Gamma(\mathcal{D}^T), \  Z\in \Gamma(\mathcal{D}^\bot), \ U\in\Gamma(\mathcal{D}^\theta) $ and for some smooth function $\mu$
  on $M$ satisfying $Y(\mu)=0$ for any $Y\in \Gamma(\mathcal{D}^T)$.
\end{corollary}
\section{Generalized inequalities on warped product skew CR-submanifolds}
In this section, we establish two inequalities on a warped product skew CR-submanifold $M=M_2\times_f M_T$ of a Kenmotsu manifold $\bar{M}$ such that
$M_2=M_\perp\times M_\theta$. We take $dim M_T=2p$, $dim M_\perp=q$, $dim M_\theta=2s+1$ and their corresponding tangent spaces are $\mathcal{D}^T$, $\mathcal{D}^\perp$ and $\mathcal{D}^\theta\oplus\{\xi\}$ respectively.

Assume that $\{e_1,e_2,\cdots,e_p,e_{p+1}=\phi e_1,\cdots,e_{2p}=\phi e_p\}$,\\
$\{e_{2p+1}=e^{*}_1,\cdots,e_{2p+q}=e_q^*\}$ and $\{e_{2p+q+1}=\hat{e}_1,e_{2p+q+2}=\hat{e}_2,\cdots,e_{2p+q+s}=\hat{e}_s,e_{2p+q+s+1}=\hat{e}_{s+1}=\sec\theta P\hat{e}_1,\cdots,e_{2p+q+2s}=\hat{e}_{2s}=\sec\theta P\hat{e}_s,e_{2p+q+2s+1}=\hat{e}_{2s+1}=\xi\}$ are local orthonormal frames of $\mathcal{D}^T$, $D^\bot$ and $D^\theta \oplus \{\xi\}$  respectively.\\
Then the local orthonormal frames for $\phi D^\bot$, $QD^\theta$ and $\nu$ are $\{e_{n+1}=\tilde{e_1}=\phi e_1^*,\cdots,e_{n+q}=\tilde{e_q}=\phi e_q^*\}$, $\{e_{n+q+1}=\tilde{e}_{q+1}=\csc \theta Q \hat{e}_1,\cdots,e_{n+q+s}=\tilde{e}_{q+s}=\csc \theta Q\hat{e}_s,e_{n+q+s+1}=\tilde{e}_{q+s+1}=\csc\theta \sec\theta QP\hat{e}_1,\cdots,e_{n+q+2s}=\tilde{e}_{q+2s}=\csc\theta\sec \theta Q P\hat{e}_s \}$ and $\{e_{n+q+2s+1},\cdots,e_{2m+1}\}$, respectively.
Clearly $dim\ \nu=(2m+1-n-q-2s)$. \\
Now, we have the following inequalities:
\begin{theorem}
Let $M=M_2\times_fM_T$ be a warped product skew $CR$-submanifold of order 1 of a Kenmotsu manifold $\bar{M}$ such that $\xi$ is tangent to $M_\theta$, where $M_2=M_\bot\times M_\theta$, then the squared norm of the second fundamental form satisfies
\begin{equation}\label{6.1}
\|h\|^2\geq 2p[\parallel \boldsymbol{\nabla}^\bot \ln f \parallel^2+(\csc^2\theta+\cot^2\theta)\{\parallel \boldsymbol{\nabla}^\theta \ln f \parallel^2-1\}],
\end{equation}
where $\boldsymbol{\nabla}^\bot\ln f$ and $\boldsymbol{\nabla}^\theta\ln f$ are the gradient of $\ln f$ along $M_\bot$ and $M_\theta$, respectively and for the case of equality, $M_2$ becomes totally geodesic and $M_T$ becomes totally umbilical in $\bar{M}$.
\end{theorem}
\begin{proof}
From (\ref{2.8}), we have
\begin{equation*}
\|h\|^2=\sum_{i,j=1}^{n} g(h(e_i,e_j), h(e_i,e_j))=\sum_{r=n+1}^{2m+1} g(h(e_i,e_j), e_r)^2.
\end{equation*}
Decomposing the above relation for our constructed frames, we get
\begin{eqnarray}\label{6.2}
&&\|h\|^2= \sum_{r=n+1}^{2m+1}\sum_{i,j=2p+1}^{q}g(h(e_i^*,e_j^*),e_r)^2+2\sum_{r=n+1}^{2m+1}\sum_{i=2p+1}^{q} \sum_{j=1}^{2s+1} g(h(e_i^*,\hat{e_j}),e_r)^2 \\
\nonumber&&+\sum_{r=n+1}^{2m+1}\sum_{i,j=1}^{2s+1}g(h(\hat{e_i},\hat{e_j}),e_r)^2 +2\sum_{r=n+1}^{2m+1}\sum_{i=1}^{q} \sum_{j=1}^{2p} g(h(e_i^*,e_j),e_r)^2 \\
\nonumber &&+2\sum_{r=n+1}^{2m+1}\sum_{i=1}^{2s+1} \sum_{j=1}^{2p} g(h(\hat{e_i},e_j),e_r)^2+\sum_{r=n+1}^{2m+1}\sum_{i,j=1}^{2p} g(h(e_i,e_j),e_r)^2.
\end{eqnarray}
Now, again decomposing (\ref{6.2}) along the normal subbundles $\phi \mathcal{D}^\bot$, $Q\mathcal{D}^\theta$ and $\nu$, we get
\begin{eqnarray}\label{6.3}
&& \qquad  \|h\|^2=\sum_{r=n+1}^{n+q}\sum_{i,j=2p+1}^{q}g(h(e_i^*,e_j^*),e_r)^2 \\
\nonumber && \qquad +\sum_{r=n+q+1}^{n+2s}\sum_{i,j=2p+1}^{q}g(h(e_i^*,e_j^*),e_r)^2+\sum_{r=n+2s+1}^{2m+1}\sum_{i,j=2p+1}^{q}g(h(e_i^*,e_j^*),e_r)^2 \\
\nonumber&& \qquad +2\sum_{r=n+1}^{n+q}\sum_{i=2p+1}^{q}\sum_{j=1}^{2s+1} g(h(e_i^*,\hat{e}_j),e_r)^2 +2\sum_{r=n+q+1}^{n+2s}\sum_{i=2p+1}^{q}\sum_{j=1}^{2s+1} g(h(e_i^*,\hat{e}_j),e_r)^2 \\
\nonumber&& \qquad +2\sum_{r=n+2s+1}^{2m+1}\sum_{i=2p+1}^{q}\sum_{j=1}^{2s+1} g(h(e_i^*,\hat{e}_j),e_r)^2+\sum_{r=n+1}^{n+q}\sum_{i,j=1}^{2s+1}g(h(\hat{e}_i,\hat{e}_j),e_r)^2\\
\nonumber && \qquad +\sum_{r=n+q+1}^{n+2s}\sum_{i,j=1}^{2s+1}g(h(\hat{e}_i,\hat{e}_j),e_r)^2+\sum_{r=n+2s+1}^{2m+1}\sum_{i,j=1}^{2s+1}g(h(\hat{e}_i,\hat{e}_j),e_r)^2
\end{eqnarray}
\begin{eqnarray}
\nonumber&&\qquad +2\sum_{r=n+1}^{n+q}\sum_{i=1}^{q}\sum_{j=1}^{2p} g(h({e}_i^*,{e}_j),e_r)^2+2\sum_{r=n+q+1}^{n+2s}\sum_{i=1}^{q}\sum_{j=1}^{2p} g(h({e}_i^*,{e}_j),e_r)^2 \\
\nonumber&&\qquad +2\sum_{r=n+2s+1}^{2m+1}\sum_{i=1}^{q}\sum_{j=1}^{2p} g(h({e}_i^*,{e}_j),e_r)^2+2\sum_{r=n+1}^{n+q}\sum_{i=1}^{2s+1}\sum_{j=1}^{2p} g(h(\hat{e}_i,{e}_j),e_r)^2 \\
\nonumber&& \qquad+2\sum_{r=n+q+1}^{n+2s}\sum_{i=1}^{2s+1}\sum_{j=1}^{2p} g(h(\hat{e}_i,{e}_j),e_r)^2+2\sum_{r=n+2s+1}^{2m+1}\sum_{i=1}^{2s+1}\sum_{j=1}^{2p} g(h(\hat{e}_i,{e}_j),e_r)^2 \\
\nonumber && \qquad+\sum_{r=n+1}^{n+q}\sum_{i,j=1}^{2p} g(h({e}_i,{e}_j),e_r)^2+\sum_{r=n+q+1}^{n+2s}\sum_{i,j=1}^{2p}g(h(e_i,e_j),e_r)^2 \\
\nonumber && \qquad  +\sum_{r=n+2s+1}^{2m+1}\sum_{i,j=1}^{2p}g(h(e_i,e_j),e_r)^2.
\end{eqnarray}
Now, by Proposition 4.1, the tenth, eleventh, thirteenth and fourteenth terms of (\ref{6.3}) are equal to zero. Also, we can not find any relation for a warped product in the form $g(h(E,F),\nu)$ for any $E,F\in \Gamma(TM)$. So, leaving the positive third, sixth, ninth, twelfth, fifteenth and eighteenth terms of (\ref{6.3}). Thus, we get
\begin{eqnarray}\label{6.4}
\nonumber \|h\|^2&\geq& \sum_{r=1}^{q}\sum_{i,j=2p+1}^{q}g(h(e_i^*,e_j^*),\phi e_r^*)^2+\sum_{r=1}^{2s}\sum_{i,j=2p+1}^{q}g(h(e_i^*,e_j^*),\tilde{e}_r)^2  \\
   &+&2\sum_{r=1}^{q}\sum_{i=2p+1}^{q}\sum_{j=1}^{2s+1}g(h(e_i^*,\hat{e}_j),\phi e_r^*)^2+2\sum_{r=1}^{s}\sum_{i=2p+1}^{q}\sum_{j=1}^{2s+1}g(h(e_i^*,\hat{e}_j),\tilde{e}_r)^2 \\
  \nonumber &+& \sum_{r=1}^{q}\sum_{i,j=1}^{2s+1}g(h(\hat{e}_i,\hat{e}_j),\phi e_r^*)^2+\sum_{r=1}^{s}\sum_{i,j=1}^{2s+1}g(h(\hat{e}_i,\hat{e}_j),\tilde{e}_r)^2  \\
  \nonumber&+& \sum_{r=1}^{q}\sum_{i,j=1}^{2p}g(h(e_i,e_j),\phi e_r^*)^2+\sum_{r=1}^{2s}\sum_{i,j=1}^{2p}g(h(e_i,e_j),\tilde{e}_r)^2.
\end{eqnarray}
Also, we have no relation for a warped product of the forms $g(h(Z,W),\phi\mathcal{D}^\bot)$, $g(h(Z,W), Q\mathcal{D}^\theta)$, $g(h(Z,U),\phi \mathcal{D}^\bot)$, $g(h(Z,U), Q\mathcal{D}^\theta)$, $g(h(U,V),\phi Z)$ and $g(h(U,V), Q\mathcal{D}^\theta)$ for any $Z,\ W\in \Gamma(\mathcal{D}^\bot)$, $U\in \Gamma(\mathcal{D}^\theta\oplus \{\xi\})$. So, we leave these terms from (\ref{6.4}) and obtain
\begin{equation}\label{6.5}
 \|h\|^2\geq \sum_{r=1}^{q}\sum_{i,j=1}^{2p}g(h(e_i,e_j),\phi e_r^*)^2+\sum_{r=1}^{2s}\sum_{i,j=1}^{2p}g(h(e_i,e_j),\tilde{e}_r)^2.
\end{equation}
Now,
\begin{eqnarray*}
  \sum_{r=1}^{q}\sum_{i,j=1}^{2p}g(h(e_i,e_j),\phi e_r^*)^2 &=& \sum_{r=1}^{q}\sum_{i=1}^{p}g(h(e_i,\phi e_j), \phi e_r^*)^2+\sum_{r=1}^{q}\sum_{i=1}^{p}g(h(\phi e_i,e_j), \phi e_r^*)^2 \\
  \nonumber&+&\sum_{r=1}^{q}\sum_{i=1}^{p}g(h(\phi e_i,\phi e_j), \phi e_r^*)^2.
\end{eqnarray*}
Using Corollary 4.1$\big((i)\ and \ (ii)\big)$, the above relation reduces to
\begin{eqnarray}\label{6.6}
\sum_{r=1}^{q}\sum_{i,j=1}^{2p}g(h(e_i,e_j),\phi e_r^*)^2&=& 2\sum_{r=1}^{q}\sum_{i=1}^{p}g(h(\phi e_i,e_j),\phi e_r^*)^2 \\
\nonumber &+&2\sum_{r=1}^{q}\sum_{i=1}^{p}g(h(e_i,e_j),\phi e_r^*)^2.
\end{eqnarray}
By virtue of (\ref{4.11}), (\ref{6.6}) yields
\begin{eqnarray}\label{6.7}
  \sum_{r=1}^{q}\sum_{i,j=1}^{2p}g(h(e_i,e_j),\phi e_r^*)^2 &=& 2\sum_{r=1}^{q}\sum_{i,j=1}^{p}\{\eta(e_r^*)-e_r^*\ln f\}^2g(e_i,e_j)^2 \\
  \nonumber &+& 2\sum_{r=1}^{q}\sum_{i,j=1}^{p}\{\eta(e_r^*)-e_r^*\ln f\}^2g(e_i,\phi e_j)^2.
\end{eqnarray}
Now, since $\eta(e_r^*)=0$ for every $r=1,2,\cdots,q$ and $g(e_i,\phi e_j)=0$ for every $i,j=1,2,\cdots,p$ so (\ref{6.7}) turns into
\begin{equation}\label{6.8}
\sum_{r=1}^{q}\sum_{i,j=1}^{2p}g(h(e_i,e_j),\phi e_r^*)^2=2p\sum_{r=1}^{q}(e_r^*\ln f)^2=2p\|\boldsymbol{\nabla}^\bot \ln f\|^2.
\end{equation}
On the other hand,
\begin{eqnarray*}
  &&\sum_{r=1}^{2s}\sum_{i.j=1}^{2p}g(h(e_i,e_j),\tilde{e_r})^2\\ &&= \csc^2\theta\sum_{r=1}^{s}\sum_{i,j=1}^{2p}g(h(e_i,e_j),Q\hat{e}_r)^2
 +\sec^2\theta\csc^2\theta\sum_{r=1}^{s}\sum_{i,j=1}^{2p}g(h(e_i,e_j),QP\hat{e_r})^2  \\
  \nonumber\qquad &&=\csc^2\theta\sum_{r=1}^{s}\sum_{i,j=1}^{p}g(h(e_i,e_j),Q\hat{e}_r)^2+\csc^2\theta\sum_{r=1}^{s}\sum_{i,j=1}^{p}g(h(e_i,\phi e_j),Q\hat{e}_r)^2  \\
  \nonumber && +\csc^2\theta\sum_{r=1}^{s}\sum_{i,j=1}^{p}g(h(\phi e_i,e_j),Q\hat{e}_r)^2+\csc^2\theta\sum_{r=1}^{s}\sum_{i,j=1}^{p}g(h(\phi e_i,\phi e_j),Q\hat{e}_r)^2 \\
  \nonumber && +\sec^2\theta\csc^2\theta\sum_{r=1}^{s}\sum_{i,j=1}^{p}g(h(e_i,e_j),QP\hat{e}_r)^2+\sec^2\theta\csc^2\theta\sum_{r=1}^{s}\sum_{i,j=1}^{p}g(h(e_i,\phi e_j),QP\hat{e}_r)^2 \\
  \nonumber && +\sec^2\theta\csc^2\theta\sum_{r=1}^{s}\sum_{i,j=1}^{p}g(h(\phi e_i,e_j),QP\hat{e}_r)^2+\sec^2\theta\csc^2\theta\sum_{r=1}^{s}\sum_{i,j=1}^{p}g(h(\phi e_i,\phi e_j),QP\hat{e}_r)^2.
\end{eqnarray*}
Using Corollary 4.1, $\big((iii)\ and \ (iv)\big)$, (\ref{4.12}), (\ref{4.13}) and the fact that $g(e_i,\phi e_j)=0$ for every $i,j=1,2,\cdots,p$ in the above relation, we obtain
\begin{eqnarray*}
  &&\sum_{r=1}^{2s}\sum_{i,j=1}^{2p}g(h(e_i,e_j),\tilde{e}_r)^2 =2p\csc^2\theta\sum_{r=1}^{s}(P\hat{e}_r\ln f)^2+2p\csc^2\theta\sum_{r=1}^{s}\{\eta(\hat{e}_r)-(\hat{e}_r\ln f)\}^2  \\
  \nonumber &&  \\
  \nonumber &&\qquad +2p\sec^2\theta\csc^2\theta\cos^4\theta\sum_{r=1}^{s}\{\eta(\hat{e}_r)-(\hat{e}_r\ln f)\}^2 + 2p\sec^2\theta\csc^2\theta\sum_{r=1}^{s}(P\hat{e}_r\ln f)^2 \\
  \nonumber&&
\end{eqnarray*}
Since $\eta(\hat{e}_r)=0$ for every $r=1,2,\cdots,s$, the above equation reduces to
\begin{eqnarray*}
  \sum_{r=1}^{2s}\sum_{i,j=1}^{2p}g(h(e_i,e_j),\tilde{e}_r)^2 &=& 2p\cot^2\theta\sum_{r=1}^{s}(\sec\theta P\hat{e}_r\ln f)^2+2p\csc^2\theta\sum_{r=1}^{s}(\hat{e}_r\ln f)^2 \\
   &+&2p\cot^2\theta\sum_{r=1}^{s}(\hat{e}_r\ln f)^2+2p\csc^2\theta\sum_{r=1}^{s}(\sec\theta P\hat{e}_r\ln f)^2 \\
   &=& 2p\cot^2\theta\sum_{r=1}^{2s}(\hat{e}_r\ln f)^2+2p\csc^2\theta\sum_{r=1}^{2s}(\hat{e}_r\ln f)^2 \\
   &=& 2p(csc^2\theta+\cot^2\theta)\left\{\sum_{r=1}^{2s+1}(\hat{e}_r\ln f)^2-(\xi \ln f)^2\right\}.
\end{eqnarray*}
Using (\ref{2.10}) and (\ref{4.3}) in the above equation, we get
\begin{equation}\label{6.9}
\sum_{r=1}^{2s}\sum_{i,j=1}^{2p}g(h(e_i,e_j),\tilde{e}_r)^2=2p(\csc^2\theta+\cot^2\theta)\{\|\boldsymbol{\nabla}^\theta \ln f\|^2-1\}.
\end{equation}
using (\ref{6.8}) and (\ref{6.9}) in (\ref{6.5}), we get the inequality (\ref{6.1}).\\ If the inequality of (\ref{6.1}) holds, then by leaving third term of (\ref{6.3}), we get $g(h(\mathcal{D}^\bot,\mathcal{D}^\bot),\nu)=0$, which implies that
\begin{equation}\label{6.10}
h(\mathcal{D}^\bot,\mathcal{D}^\bot)\perp \nu.
\end{equation}
Also, by leaving the first and second term of (\ref{6.4}), we get $h(\mathcal{D}^\bot,\mathcal{D}^\bot)\perp \phi \mathcal{D}^\bot$ and $h(\mathcal{D}^\bot,\mathcal{D}^\bot)\perp Q\mathcal{D}^\theta$ respectively. Therefore
\begin{equation}\label{6.11}
h(\mathcal{D}^\bot,\mathcal{D}^\bot)\subseteq \nu.
\end{equation}
From (\ref{6.10}) and (\ref{6.11}), we obtain
\begin{equation}\label{6.12}
  h(\mathcal{D}^\bot,\mathcal{D}^\bot)=0.
\end{equation}
Similarly by leaving sixth term of (\ref{6.3}), we get
\begin{equation}\label{6.13}
h(\mathcal{D}^\bot,\mathcal{D}^\theta)\perp \nu.
\end{equation}
Also, leaving the third and fourth term of (\ref{6.4}), we get $h(\mathcal{D}^\bot,\mathcal{D}^\theta)\perp \phi \mathcal{D}^\bot$ and $h(\mathcal{D}^\bot,\mathcal{D}^\theta)\perp Q \mathcal{D}^\theta$ respectively. Therefore,
\begin{equation}\label{6.14}
h(\mathcal{D}^\bot,\mathcal{D}^\theta)\subseteq \nu.
\end{equation}
From (\ref{6.13}) and (\ref{6.14}), we obtain
\begin{equation}\label{6.15}
h(\mathcal{D}^\bot,\mathcal{D}^\theta)=0.
\end{equation}
Again, by leaving ninth term of (\ref{6.3}), we get
\begin{equation}\label{6.16}
h(\mathcal{D}^\theta,\mathcal{D}^\theta)\perp \nu.
\end{equation}
Also, leaving fifth and sixth term of (\ref{6.4}), we get $h(\mathcal{D}^\theta,\mathcal{D}^\theta)\perp \phi \mathcal{D}^\bot $ and \\ $h(\mathcal{D}^\theta,\mathcal{D}^\theta)\perp Q \mathcal{D}^\theta $ respectively. Therefore,
\begin{equation}\label{6.17}
h(\mathcal{D}^\theta,\mathcal{D}^\theta)\subseteq \nu.
\end{equation}
From (\ref{6.16}) and (\ref{6.17}), we obtain
\begin{equation}\label{6.18}
h(\mathcal{D}^\theta,\mathcal{D}^\theta)=0.
\end{equation}
Next by leaving the twelfth term of (\ref{6.3}), we get
\begin{equation}\label{6.19}
h(\mathcal{D}^\bot,\mathcal{D}^T)\perp \nu.
\end{equation}
From (\ref{6.19}) and Proposition 4.1, we get
\begin{equation}\label{6.20}
h(\mathcal{D}^\bot,\mathcal{D}^T)=0.
\end{equation}
Also, leaving fifteenth term of (\ref{6.3}), we get
\begin{equation}\label{6.21}
h(\mathcal{D}^\theta,\mathcal{D}^T)\perp \nu.
\end{equation}
From (\ref{6.21}) and Proposition 4.1, we get
\begin{equation}\label{6.22}
h(\mathcal{D}^\theta,\mathcal{D}^T)=0.
\end{equation}
Thus from (\ref{6.12}), (\ref{6.15}), (\ref{6.18}), (\ref{6.20}), (\ref{6.22}) and the fact that $M_2$ is totally geodesic in $M$ (\cite{BISHOP},\cite{CHENCR1}), we conclude that
$M_2$ is totally geodesic in $\bar{M}$. Next by leaving the eighteenth term of (\ref{6.3}), we get
\begin{equation}\label{6.23}
h(\mathcal{D}^T,\mathcal{D}^T)\perp \nu.
\end{equation}
Then from (\ref{5.16}), (\ref{5.18}), (\ref{6.23}) and the fact that $M_T$ is totally umbilical in $M$ (\cite{BISHOP},\cite{CHENCR1}), we conclude that $M_T$ is totally umbilical in $\bar{M}$. This completes the proof of the theorem.
\end{proof}
\begin{theorem}
  Let $M=M_2\times_f M_T$ be a warped product skew CR-submanifold of order $1$ of a Kenmotsu manifold $\bar{M}$ such that $\xi$ is tangential to $M_\perp$,
  where $M_2=M_\perp\times M_\theta$, then  the squared norm of the second fundamental form satisfies
\begin{equation}\label{6.24}
\|h\|^2\geq 2p[\|\boldsymbol{\nabla}^\bot \ln f\|^2-1+(\csc^2\theta+\cot^2\theta)\|\boldsymbol{\nabla}^\theta \ln f\|^2].
\end{equation}
If the equality of (\ref{6.24}) holds, then $M$ is totally geodesic and $M_T$ is totally umbilical in $\bar{M}$.
\end{theorem}
\begin{proof}
For this theorem, we take $dim \ M_\perp=q+1$ and $dim \ M_\theta=2s$. So, orthonormal frames of $\mathcal{D}^\bot\oplus \{\xi\}$ and $\mathcal{D}^\theta$ will be
$\{e_{2p+1}=e_1^*,\cdots,e_{2p+q}=e_q^*,e_{2p+q+1}=\xi\}$\\ and
$\{e_{2p+q+2}=\hat{e}_1,\cdots,e_{2p+q+s+1}=\hat{e}_s, e_{2p+q+s+2}=\hat{e}_{s+1}=\sec\theta P\hat{e}_1,\\ \cdots,e_{2p+q+2s+1}=\hat{e}_{2s}=\sec\theta P\hat{e}_s\}$, respectively. Then the proof of the theorem is similar as Theorem 6.1.
\end{proof}
\noindent\textbf{Remark:} If we take $dim \ M_\theta=0$ in a warped product skew CR-submanifold $M=M_2\times_f M_T$ of a Kenmotsu manifold $\bar{M}$ such that
$M_2=M_\perp\times M_\theta$, then it turns into CR-warped product $M=M_\bot\times_fM_T$ which was studied in \cite{UAN}. Therefore, Theorem 5.1 and Theorem 6.2 are the generalizations of results of \cite{UAN} as follows:
\begin{corollary}\emph{(Theorem 3.1 of \cite{UAN})}
  A proper contact CR-submanifold of a Kenmotsu manifold $\bar{M}$ is locally a contact CR-warped product of the form $M_\perp\times_fM_T$ if and only if
\begin{equation*}
A_{\phi Z}X=\{\eta(Z)-(Z\mu)\}\phi X,
\end{equation*}
for every $X\in \Gamma(\mathcal{D}^T)$ and $Z\in \Gamma(\mathcal{D}^\bot\oplus\{\xi\})$, for some function $\mu$ on $M$ satisfying $(Y\mu)=0$ for any $Y\in \Gamma(\mathcal{D}^T)$.
\end{corollary}
\begin{corollary}\emph{(Theorem 3.2 of \cite{UAN})}
Let $\bar{M}$ be a $(2m+1)$-dimensional Kenmotsu manifold and $M=M_\perp\times_fM_T$ an $n$-dimensional contact CR-warped product submanifold, such that $M_\perp$ is a $(q+1)$-dimensional anti-invariant submanifold tangent to $\xi$ and $M_T$ is a $2p$-dimensional invariant submanifold of $\bar{M}$, then
the squared norm of the second fundamental form of $M$ satisfies
\begin{equation}\label{6.25}
\|h\|^2\geq 2p[\|\boldsymbol{\nabla}^\bot \ln f\|^2-1]
\end{equation}
where $\boldsymbol{\nabla}^\bot \ln f$ is the gradient of $\ln f$.
 If the equality of (\ref{6.25}) holds, then $M_\perp$ is totally geodesic and $M_T$ is totally umbilical in $\bar{M}$.
\end{corollary}
\noindent\textbf{Acknowledgement:} The first and third authors (SKH and JR) gratefully acknowledges to the SERB (Project No: EMR/2015/002302), Govt. of India for financial assistance of the work.

\vspace{0.1in}
\noindent S. K. Hui$^1$, T. Pal$^2$ and J. Roy$^3$\\
Department of Mathematics, The University of Burdwan, Burdwan, 713104, West Bengal, India.\\
E-mail: skhui@math.buruniv.ac.in$^1$;\\
tanumoypalmath@gmail.com$^2$; joydeb.roy8@gmail.com$^3$\\
\end{document}